\documentclass[12pt]{article}
\usepackage{graphicx,amsmath,amssymb,amsthm,rotating,stmaryrd,datetime,pifont,booktabs,enumitem}
\usepackage[mathscr]{eucal}

\newtheorem{theorem}{Theorem}

\newtheorem{result}[theorem]{Result}
\newtheorem{lemma}[theorem]{Lemma}

\newtheorem{construction}[theorem]{Construction}

\usepackage{setspace}
\usepackage{enumitem}
\setlist{nolistsep}
\setlist{nosep}

\usepackage[usenames,dvipsnames]{color}
\newcommand\red[1]{{\color{red} #1}}

\renewcommand{\O}{\mathcal O}
\newcommand{\A}{\mathcal A}

\renewcommand{\P}{\mathcal{P}}
\renewcommand{\S}{\mathcal{S}}

\newcommand{\K}{\mathcal{K}}
\newcommand{\N}{\mathcal{N}}

\newcommand{\R}{\mathcal{R}}
\newcommand{\C}{\mathcal{C}}
\newcommand{\D}{\mathcal{D}}

\newcommand{\li}{\ell_\infty}
\newcommand{\si}{\Sigma_\infty}

\newcommand{\F}{\mathcal{F}}
\newcommand{\PG}{{\textup{PG}}}
\newcommand{\AG}{{\textup{AG}}}

\newcommand{\PGammaL}{\textup{P}\Gamma{\textup {L}}}
\newcommand{\Fqq}{\mathbb{F}_{q^2}}

\newcommand{\Fqt}{\mathbb{F}_{\hspace*{-1mm}{q^{t}}}}
\newcommand{\Fq}{\mathbb{F}_{\hspace*{-.5mm}q}}

\DeclareMathOperator{\Nm}{N}

\begin{document}

\title{Inherited  conics in  Andr\'e planes of odd order}

\author{S.G. Barwick\footnote{(corresponding author).  School of Mathematical Sciences, Adelaide University, Adelaide, 5005, Australia. susan.barwick@adelaide.edu.au}, 
Alice M.W. Hui\footnote{School of Mathematical and Statistical Sciences, Clemson University, USA. aliceh2@clemson.edu}
 and Wen-Ai Jackson\footnote{School of Mathematical Sciences, Adelaide University, Adelaide, 5005, Australia. wen.jackson@adelaide.edu.au}
 }
 
 \date{}
\maketitle



Dedicated to the memory of Philip P. W. Wong (1953–2025).


\begin{abstract} The process of deriving the Desarguesian plane $\PG(2,q^2)$ to get the Hall plane is well known, and the problem of when a conic in $\PG(2,q^2)$  inherits to an arc in the Hall plane has been solved. In this article we look at the generalisation of replacing an Andr\'e net of  $\PG(2,q^t)$, $t\geq 3$ to construct an Andr\'e plane of order $q^t$. This article looks at the case where $q$ is odd and $t$ is prime, and determines when a conic in $\PG(2,q^t)$ inherits to an arc in an Andr\'e plane. We first consider  replacing a single Andr\'e net, then generalise to replacing multiple Andr\'e nets, and obtain results in Andr\'e and subregular planes. 
Further, the number of arcs   that are inherited from a single Andr\'e replacement is determined.
 \end{abstract}

\section{Introduction}

A \emph{$k$-arc} in a projective plane of order $q$ is a set of $k$ points, no three collinear. A $k$-arc is \emph{complete} if it is not contained in a $(k+1)$-arc. An arc has at most $q+1$ points if $q$ is odd, and at most $q+2$ points if $q$ is even, see \cite{bose47}.
A $(q+1)$-arc is called an \emph{oval}. Arcs and ovals are well studied in $\PG(2,q)$, and an interesting question is the existence of ovals in non-Desarguesian planes. A well studied theme has been starting with a conic in the Desarguesian plane $\PG(2,q^2)$ and using the process of derivation to construct inherited arcs in the Hall plane. Results on the effect of derivation on conics of $\PG(2,q^2)$ have been obtained in \cite{barmar}, \cite{BKNS}, \cite{Cher}, \cite{GlynnStein}, \cite{Korch1}, \cite{Korch2}, \cite{OKeefePasPen}, \cite{OKeefePas}, \cite{Szon}.
In this article we are interested in the case where $q$ is odd.
The complete results for the Hall plane when $q$ is odd are summarised in the next result.
We use the following terminology. If $\C$ is a conic in $\PG(2,q^t)$, then we say  $\C$ is  secant, tangent or exterior  to the line at infinity $\li$ if $\C$ meets $\li$ in two, one or zero points respectively. Further, as $q$ is odd, a point of $\PG(2,q^t)\setminus \C$ is \emph{external} or \emph{internal} to $\C$ if it lies on two or zero tangents of $\C$ respectively, see \cite[Sec 8.2]{Hirsch1}.

\begin{result}\label{120}
Let $\C$ be a non-degenerate conic in $\PG(2,q^2)$, $q$ odd, $q>5$, and $\D$ a Baer subline of $\li$. Derive $\PG(2,q^2)$ using $\D$ to get the Hall plane ${\mathcal H}(\D)$. The affine points of $\C$ correspond to a set of affine points of ${\mathcal H}(\D)$ denoted by $\D(\C)$.
\begin{enumerate}
\item If $\C$ is tangent to $\li$, then $\D(\C)$ is not an arc of the Hall plane.
\item If $\C$ is exterior to $\li$, then $\D(\C)$ is not an arc of the Hall plane.
\item Suppose $\C$ is secant to $\li$ with $\C\cap\li=\{P,Q\}$.
\begin{enumerate}
\item   If $P,Q\not\in\D$, $P,Q$ are conjugate with respect to $\D$,  and the points in $\D$ are all internal to $\C$, then $\D(\C)$ is an arc of the Hall plane that can be completed to an oval.
\item If $P,Q\in\D$ and the points in $\D\setminus\{P,Q\}$ are all internal to $\C$, then $\D(\C)$ is a complete $(q^2-1)$-arc of the Hall plane.
\item  Otherwise $\D(\C)$ is not an arc of the Hall plane.
\end{enumerate}
\end{enumerate}
\end{result}

The process of deriving $\PG(2,q^2)$ using a Baer subline of $
\li$  to get the Hall plane can be generalised to performing an Andr\'e replacement of $\PG(2,q^t)$, $t\geq3$, using an Andr\'e set of $\li$ to get an Andr\'e plane of order $q^t$ (see Section \ref{sec:back}). This article considers whether conics in $\PG(2,q^t)$, $t\geq 3$ prime, $q$ odd, inherit to arcs in Andr\'e planes. The main results of the article are given in Theorems~\ref{200} and \ref{200G}; these results exhibit a similar pattern to Result~\ref{120}.


The article is set out as follows. Section 2 describes the process of Andr\'e replacement. We use a set of points on $\li$ that we call an Andr\'e net -- this set has been considered in several different settings, and we give a short survey of the different names used for this set. While the process of Andr\'e replacement is well known,  we give a careful description of it in Section 2 in order to establish the notation, and to explain why the article focuses on the case where $t$ is prime.
Section 3 proves some preliminary results using the equivalence between Andr\'e sets and scattered $\Fq$-linear sets of pseudoregulus type. 

Section 4 proves the main results, namely  Theorems~\ref{200} and \ref{200G}, which completely determine when a conic of $\PG(2,q^t$), $q$ odd, $t\geq3$ prime, inherits to an arc of an Andr\'e plane (and more generally, a subregular plane).  
Finally, Section 5 does some counting of inherited arcs:  Theorem~\ref{119} counts arcs in the Hall plane  that are inherited from conics of $\PG(2,q^2)$, $q$ odd; and
Theorem \ref{118t} counts  arcs in an   Andr\'e plane of index 1 that are inherited from conics of $\PG(2,q^t)$, $q$ odd, $t\geq3$ prime.

The case when $q$ is even is considered in a companion paper \cite{BHJ-even}, where a partial answer is given to the question of when conics in $\PG(2,q^t)$, $q$ even, $t$ prime,  inherit to arcs in Andr\'e planes.

\section{Background}\label{sec:back}

\subsection{The Bruck-Bose representation}

The finite field of prime power order $q$ is denoted $\Fq$, and we let $\Fq^*=\Fq\setminus\{0\}$.
The Bruck-Bose representation \cite{andr54,bruc69,bruc64,segre} of $\PG(2,q^t)$ in $\PG(2t,q)$  is well known. Let $\si$ be the hyperplane at infinity in $\PG(2t,q)$ and let $\S$ be a   $(t-1)$-spread in $\si$. The incidence structure $\A(\S)$ with points the points of $\PG(2t,q)\setminus\si$; lines the $t$-spaces of $\PG(2t,q)\setminus\si$ that contain an element of $\S$; and incidence being inclusion, is an affine plane. We can complete this to a projective plane denoted $\P(\S)$; points on the line at infinity $\li$ correspond to the elements of $\S$. Moreover, $\P(\S)\cong \PG(2,q^t)$ iff $\S$ is a regular (Desarguesian) spread.
If $\mathcal X$ is a set of points of $\PG(2,q^2)$, then we denote the corresponding set of points in $\PG(2t,q)$ by $[\mathcal X]$.

We will need to use coordinates for points on the line
$\li\cong\PG(1,q^t)$ and the corresponding spread elements in $\si\cong\PG(2t-1,q)$.
A point $P$ in $\PG(1,q^t)$  has homogeneous coordinates  $P=(x,y)$ for some $x,y\in\Fqt$, not both $0$.   This is an element of the vector space $V(2,q^t)$, so we can convert it into a vector in $V(2t,q)$, a process which is commonly called \emph{field reduction} and denoted $\F_{2,t,q}$. That is,  $\F_{2,t,q}:V(2,q^t)\longrightarrow V(2t,q)$, and we use the notation
$\F_{2,t,q}(x,y)=\langle(x,y)\rangle_{q}$.
(For example, construct $\Fqt$ as an  extension of $\Fq$ using a primitive element  $\tau$  with primitive
polynomial of degree $n$. So
 every element $x\in\Fqt$
can be uniquely written as $x=x_0+x_1\tau+\ldots+x_{t-1}\tau^{t-1}$ with
$x_i\in\Fq$. Then  $\langle(x,y)\rangle_{q}=(x_0,x_1,\ldots,x_{t-1},\,y_0,y_1,\ldots,y_{t-1})$.) 
The point $P=(x,y)\in\PG(1,q^t)$  has homogeneous coordinates $(sx,sy)$, for $s\in\Fqt^*$, so the corresponding spread element in $\PG(2t-1,q)$ is the    $(t-1)$-space with coordinates $[P]=\{\langle (sx,sy)\rangle_q\,|\,s\in\Fqt^*\}$.

 \subsection{Andr\'e nets}

The process of deriving the Desarguesian plane $\PG(2,q^2)$ to  construct  the Hall plane of order $q^2$  is well known. Derivation involves replacing certain lines of $\PG(2,q^2)$ with certain Baer subplanes; it  can be interpreted in the $\PG(4,q)$ Bruck-Bose setting  as a net replacement as follows. Let $\R$ be a regulus contained in $\S$ (a regular $1$-spread in $\si\cong\PG(3,q)$) and let $\R'$ denote the opposite regulus. Let  $\S'=(\S\setminus\R)\cup\R'$, then $\P(\S')$ is the Hall plane.
More generally, let $\S$ be a $(t-1)$-spread in $\PG(2t-1,q)$, and suppose there is a subset $\N\subset\S$ and a set $\N'$ of pairwise disjoint  $(t-1)$-spaces such that each element of $\N$ meets each element of $\N'$ in exactly one point, then we call $\N$ a \emph{net} with \emph{replacement net} $\N'$.
 If $\N$ is a net of $\S$ with replacement net $\N'$, then the set $\S'=(\S\setminus\N)\cup\N'$ is a spread of $\PG(2t-1,q)$, we refer to this process as
  \emph{replacing  $\N$ by $\N'$}.
 We are interested in certain nets of the regular $(t-1)$-spread, called Andr\'e nets. These have been studied in various settings and we start with a short survey giving the different names used in the different settings, and relating some of the known results.

Let $E,F$ be two distinct points of $\PG(1,q^t)$ and let $G$ be the collineation group acting on points of $\PG(1,q^t)$ that fixes $E$ and $F$ pointwise. As $G$ is Singer group of order $q^t-1$, $G$ has a unique Singer subgroup $H$ of order $q^{t-1}+\cdots+q+1$. The orbits of $H$ partition the  points of $\PG(1,q^t)\setminus\{E,F\}$ into $q-1$ sets  of size $q^{t-1}+\cdots+q+1$,  we call  these sets \emph{Andr\'e sets} with \emph{transversal points} $E,F$. Andr\'e sets have been investigated in several different settings. In the $\PG(2t-1,q)$  field reduction setting, an Andr\'e set corresponds to a set of $(t-1)$-spaces contained in a regular $(t-1)$-spread of $\PG(2t-1,q)$, so forms a partial spread. Andr\'e sets  are called \emph{Andr\'e nets} in \cite{andr54} and \cite{ostrom70}; and \emph{norm surfaces}
in \cite{bruck70}. Ostrom introduced the term   \emph{hyperregulus} in \cite{ostrom}, a    term which is used in   the translation plane literature, see \cite{john07}, and Andr\'e sets correspond  to  \emph{Andr\'e hyperreguli}.
In \cite{bruc73a,bruc73b}, Bruck developed the theory of circle geometries $C(t,q)$, $t$ prime, and Andr\'e sets are equivalent to \emph{covers of a circle geometry} with \emph{carriers $E,F$}.

 More recently, Andr\'e sets have been investigated in the linear set literature,  see \cite{lavr14a} for a survey article on linear sets.
 Let  $\Pi_k$ be a subspace of dimension $k$ in $\PG(2t-1,q)$, then the set $\{P\in\PG(1,q^t)\,|\, [P]\cap\Pi_k\neq\emptyset\}$ is called an \emph{$\Fq$-linear set of rank $k+1$}. If this set has size $q^{k}+\ldots+q+1$, then it is called a \emph{scattered} $\Fq$-linear set  of rank $k+1$.
Andr\'e sets are  scattered $\Fq$-linear sets of $\PG(1,q^t)$ of rank $t$.    The notion of an $\Fq$-linear set of $\PG(1,q^t)$ of  pseudoregulus type was introduced in \cite{LMPT14}, and Andr\'e sets are equivalent to $\Fq$-linear sets of $\PG(1,q^t)$ of  \emph{pseudoregulus type} with \emph{transversal points} $E,F$. In particular, an Andr\'e set of $\PG(1,q^t)$ is projectively equivalent to the set of points with homogeneous coordinates
 \begin{equation}
\label{e1}
\A=\{P_k=(1,k)\,|\,k\in\Fqt^*,\Nm(k)=1\}
\end{equation}
which has transversal points $E=(1,0)$, $F=(0,1)$; where $$\Nm(k)=k^{q^{t-1}+\cdots+q+1}$$ is the \emph{norm mapping} from $\Fqt$ to $\Fq$.
 Proofs of this equivalence can be found in multiple settings, for example in \cite{ostrom}, \cite[Remark 4.2]{LMPT14} and  \cite[Remark 2.2]{dona14}.
 Moreover, the $q-1$ Andr\'e sets with transversal points $(1,0)$, $(0,1)$ are the sets
  $\{(1,k)\,|\,k\in\Fqt^*,\Nm(k)=\delta\}$, $ \delta\in\Fq^*.$
If $t=3$, then every scattered $\Fq$-linear  set  of $\PG(1,q^3)$ of rank $3$ has pseudoregulus type (that is, is an  Andr\'e set), see \cite{lavr10}. If $t>3$, then there are scattered $\Fq$-linear  sets of $\PG(1,q^t)$ of rank $t$ which are not Andr\'e sets, for example, see \cite{JJ2008} in the translation plane literature, and     \cite{casa22,csaj,luna01,sheek2018} in the linear set literature, with a    survey given in \cite{grim}.

Using field reduction on the coordinates in (\ref{e1}), the Andr\'e set $\A$  corresponds to an \emph{Andr\'e net} in $\PG(2t-1,q)$, denoted $[\A]$, with coordinates given by
$$ [\A] =\{[P_k]=\{\langle(s,sk)\rangle_q\,|\, s\in\Fqt^*\}\ \,|\,k\in\Fqt^*, \Nm(k)=1\}.$$
In $\PG(2t-1,q)$, there are $t-1$ \emph{Andr\'e replacement nets} for $[\A]$, denoted $\A^i$, $i=1,\ldots,{t-1}$. Coordinates for these nets are given in multiple places in the above references, namely
\begin{equation}
\label{e2}
\A^i=\{X_{i,k}=\{\langle(s,s^{q^i}k)\rangle_q\,|\, s\in\Fqt^*\}\ \,|\,k\in\Fqt^*, \Nm(k)=1\}.
\end{equation}

\subsection{Andr\'e replacement}\label{sec2.3}

The process of one Andr\'e replacement involves replacing the regular $(t-1)$-spread $\S$ with the $(t-1)$-spread $(\S\setminus [\A])\cup\A^i$ for some $i\in\{1,\ldots,t-1\}$.
Recall that as $\S$ is regular, the plane $\P(\S)\cong \PG(2,q^t)$ is Desarguesian. The plane $\P(\S\setminus [\A]\cup\A^i)$ is an Andr\'e plane, in \cite{dover} this is called an Andr\'e plane of \emph{index} 1 (since one Andr\'e net is replaced; multiple Andr\'e replacements are discussed below). We abbreviate the notation $\P(\S\setminus [\A]\cup\A^i)$ and we denote this Andr\'e plane by  $\P(\A^i)$. The process of constructing the Andr\'e plane $\P(\A^i)$ from $\P(\S)\cong \PG(2,q^t)$ is called \emph{Andr\'e replacement}. The process was developed in  \cite{andr54} and \cite{ostrom70},  and the next statement carefully describes it.

\begin{construction}\label{cons1} {\bfseries\emph{(Andr\'e replacement)}} In $\PG(2,q^t)$, $t$ prime, $t\geq 3$, let $\D$ be an Andr\'e set   of $\li$. In the $\PG(2t,q)$ Bruck-Bose setting, $[\D]$ is an Andr\'e net of the regular spread $\S\subset\si$,   denote the $t-1$ Andr\'e replacement nets by $\D^i$, $i=1,\ldots,t-1$.
 For $i\in\{1,\ldots,t-1\}$, define  the incidence structure $\mathcal T_i$ as follows.
\begin{itemize}
\item The points of $\mathcal T_i$ are the
 points of $\AG(2,q^t)$
 \item The lines of $\mathcal T_i$ consist of
 \begin{itemize}
\item  the lines of $\AG(2,q^t)$ that meet $\li$ in a point not in $\D$,
\item the $\D^i$-blocks (a set of affine points $\mathcal X$ in $\PG(2,q^t)$ is called a \emph{$\D^i$-block} if in $\PG(2t,q)$, the corresponding set $[\mathcal X]$ is an affine $t$-space whose projective completion  meets $\si$ in a $(t-1)$-space that lies in the replacement net $\D^i$).
\end{itemize}
\item Incidence is inclusion
\end{itemize}
Then $\mathcal T_i$ is an affine plane of order $q^t$ whose projective completion is an Andr\'e plane of index 1 denoted $\P(\D^i)$. This process is called an \emph{Andr\'e replacement with respect to the Andr\'e replacement net $\D^i$. }
\end{construction}

The important property that makes Andr\'e replacement work is that each $(t-1)$-space in the Andr\'e net $[\A]$ meets each $(t-1)$-space in the Andr\'e replacement net  $\A^i$ in exactly one point. Moreover, if $t$ is prime, then elements of distinct Andr\'e replacement nets meet in exactly one point (that is,  for distinct $i,j\in\{1,\ldots,t-1\}$,
every $(t-1)$-space in $\A^i$ meets every $(t-1)$-space in $\A^j$ in exactly one point).
This is not true in general. For example, Ostrom \cite{ostrom} shows that when $t=6$, and $q$ odd, it is possible to create a non-Andr\'e replacement net for $[\A]$ that consists of a mix of some $5$-spaces of $\A^1$ and some $5$-spaces of $\A^5$. That is, there exists $5$-spaces $\alpha\in\A^1$ and $\alpha'\in\A^5$ with $\alpha\cap\alpha'=\emptyset$ (and so  there exists $\beta\in\A^1$ and $\beta'\in\A^5$ with $\beta\cap\beta'$ containing a line).  More generally,  Johnson \cite{johnson2003} shows that when $t$ is composite, it is always possible to construct a non-Andr\'e replacement net for $[\A]$.
The set of points on the spaces in $[\A]$ form a surface which is studied in \cite{lavr15} in a different context, and the dimension of $X_{i,k}\cap X_{j,h}$ from (\ref{e2})  is explicitly computed in  \cite[Prop 7]{lavr15}.
  If $q\geq t$, then \cite[Thm 14]{lavr15} shows that
the only $(t-1)$-spaces whose points are contained in the pointset covered by $[\A]$ are the $(t-1)$-spaces contained in one of the Andr\'e replacement nets, that is, those having form $X_{i,k}$ for some $i=1,\ldots,t-1$, $k\in\Fqt^*, \Nm(k)=1$.

In particular,  if $t$ is prime, then
 there are exactly $t-1$ replacement nets for the Andr\'e net $[\A]$, namely the Andr\'e replacement nets $\A^i$, $i=1,\ldots,t-1$  (and if $t$ is not prime, then there may be more than $t-1$ replacement nets).
These key properties form an important part of the proof of Theorem~\ref{M110}, and
give the motivation for considering  the case $t$ prime in this article.

More generally, we can replace more than one Andr\'e net. Let $\D_1,\ldots,\D_s$ be a set of pairwise disjoint Andr\'e sets of $\li$ in $\PG(2,q^t)$. Then we can replace each of the Andr\'e nets $[\D_j]$ with one of their $t-1$ Andr\'e replacement nets to obtain a new spread $\S'$. 
The translation plane corresponding to the new spread  $\S'$ is an Andr\'e plane of \emph{index} $s$ exactly when the Andr\'e sets $\D_1,\ldots,\D_s$ share the same transversal points (these index $s$ Andr\'e planes are called \emph{multiple Andr\'e planes} in \cite[Ch 16]{john07}). 
More generally, replacing an arbitrary set of Andr\'e nets gives a    \emph{subregular} plane (see \cite{dover,john07}).

%
%
%
%
%
%
%
%

\section{Preliminary Results}

In order to  determine when conics in $\PG(2,q^t)$  inherit to arcs of the Andr\'e plane, we need a number of preliminary results.  

\subsection{Intersection of two Andr\'e sets}

We need the following bound on the number of points contained in the intersection of two Andr\'e sets; this  is proved in \cite[Prop 3]{dover}.


\begin{result}\label{102t}
Let $\D_1$ and $\D_2$ be
two Andr\'e sets, (that is, scattered  $\Fq$-linear sets of  pseudoregulus type) of $\PG(1,q^t)$, $t\geq3$. Then $|\D_1\cap\D_2|\leq2(q^{t-2}+\cdots +q+ 1)$.
\end{result}


\subsection{$\Fq$-lines contained in an Andr\'e set}

Recall that a $(t-1)$-regulus in $\PG(2t-1,q)$ is a set $\R$ of $q+1$ pairwise disjoint $(t-1)$-spaces  with the property that if a line meets three elements of $\R$, then it meets all elements of $\R$. Hence a  $(t-1)$-regulus is  ruled by a set of $q^{t-1}+\cdots+q+1$ (pairwise disjoint) lines.  We will use the following representations in the Bruck-Bose setting, a proof can be found in \cite{BJ2012}.

 \begin{result}\label{222}
\begin{enumerate}
\item Let $b$ be an $\Fq$-line contained in $\li$ in $\PG(2,q^t)$. Then in the $\PG(2t,q)$ Bruck-Bose representation, $[b]$ is a $(t-1)$-regulus contained in the regular spread $\S$. Conversely, every $(t-1)$-regulus contained in $\S$ corresponds to an $\Fq$-line of $\li$.
\item Let $\pi$ be an $\Fq$-plane of $\PG(2,q^t)$ that is secant to $\li$. Then in the $\PG(2t,q)$ Bruck-Bose representation, $[\pi]$ is a plane that meets $\si$ in a line. Conversely, every plane in $\PG(2t,q)$ that meets $\si$ in a line corresponds to an $\Fq$-plane  of $\PG(2,q^t)$ that is secant to $\li$.
\end{enumerate}
\end{result}

Let $\D$ be an  Andr\'e set, that is, a scattered $\Fq$-linear set of pseudoregulus type of $\PG(1,q^t)$, $t\geq3$,   $q\geq t$.
The number of $\Fq$-lines contained in  $\D$ is computed in \cite{csaj16}. Of particular interest is \cite[Prop 4.9]{csaj16} which shows that if $t$ is prime, then the number of $\Fq$-lines contained in $\D$ is $$(t-1)(q^{t-1}+\ldots+q+1)\frac{q^{t-1}-1}{q^2-1}.$$ If $t$ is not prime, then it follows from \cite[Thm 4.8]{csaj16} that the number of $\Fq$-lines contained in $\D$ is strictly less than this.
Moreover when $t$ is prime, \cite[Thm 5.13]{csaj16} discusses the representation of the $\Fq$-lines contained in $\D$: under field reduction these are projections of normal rational curves. This leads to  a description of $t-1$ families of $\Fq$-lines contained in $\D$. The following result   gives a way to label these $t-1$ families in our setting, and interprets and extends results from \cite{csaj16} to detail the properties and counting we need to study inherited arcs.

\begin{theorem}\label{M110} Let $\D$ be an Andr\'e set, (that is, a scattered $\Fq$-linear set of pseudoregulus type) of $\PG(1,q^t)$, $t\geq3$ prime, $q\geq t$.
\begin{enumerate}
\item\label{M110a} The $\Fq$-lines contained in $\D$ can be partitioned into $t-1$ families of size $(q^{t-1}+\ldots+q+1)(q^{t-1}-1)/(q^2-1)$. We can label the families by $\F_1,\ldots,\F_{t-1}$ so that for $i\in\{1,\ldots,t-1\}$, family $\F_i$ corresponds to
the  Andr\'e replacement net $\D^i$  in the  following way.
\begin{enumerate}
\item If $\ell$ is a line contained in one of  the $(t-1)$-spaces in  $\D^i$, then the set $\{P\in\li\,|\,[P]\cap\ell\neq\emptyset\}$  is an $\Fq$-line belonging to family $\F_i$.
\item If
$b$ is an $\Fq$-line in family $\F_i$, then  every ruling line of the $(t-1)$-regulus $[b]$ is contained in a unique
 $(t-1)$-space of $\D^i$.
 \end{enumerate}

\item\label{M110b}  Any two distinct points of $\D$ lie in exactly $t-1$ $\Fq$-lines of $\D$, one from each family.
\item\label{M110c} Each point of   $\D$ lies in exactly
$(t-1)(q^{t-2}+\cdots+q+1)$ $\Fq$-lines of $\D$; $(q^{t-2}+\cdots+q+1)$ in each family.
 \end{enumerate}
  \end{theorem}

\begin{proof}
A point $P\in \PG(2t-1,q)$ has coordinates  $P=\langle(x,y)\rangle_q$ for some $x,y\in\Fqt$, not both $0$. For $\alpha\in\Fqt^*$,  the map $\sigma_\alpha$ defined by $\sigma_\alpha(P)=\langle(\alpha x,\alpha y)\rangle_q$ is a collineation. Let $G=\{\sigma_\alpha\,|\,\alpha\in\Fqt^*\}$.  It is straightforward  to check that $G$ fixes setwise every element of the regular $(t-1)$-spread $\S$, and acts transitively on the points of a spread element.
Without loss of generality, let $\D=\A=\{(1,\theta)\,|\,\theta\in\Fqt^*,\Nm(\theta)=1\}$ as in (\ref{e1}).
Fix $i\in\{1,\ldots,t-1\}$ and consider the Andr\'e replacement net $\A^i$ as given in (\ref{e2}).
It is straightforward to show that the collineation group $G$ acts transitively on the $(t-1)$-spaces in $\A^i$.

Let $\ell$ be a line contained in a  $(t-1)$-space   of the net $\A^i$. The line $\ell$ meets $q+1$ elements of the regular $(t-1)$-spread $\S$, denote these elements by $\R$. The
 $q^{t-1}+\cdots+q+1$ ruling lines of $(t-1)$-regulus $\R$ are the lines  $\sigma_\alpha(\ell)$, $\sigma_\alpha\in G$.
 As $G$ acts transitively on the $(t-1)$-spaces in $\A^i$,
 we have    $\sigma_\alpha(\ell)$ is a line contained in a  $(t-1)$-space of $\A^i$.
 That is, the  ruling lines of $\R$ lie one in each of the   $(t-1)$-spaces in $\A^i$.

 A $(t-1)$-space contains $x=(q^{t-1}+\ldots+q+1)(q^{t-1}-1)/(q^2-1)$ lines, so the number of lines of $\PG(2t-1,q)$ that are contained  in some  element of $\A^i$ is $x(q^{t-1}+\ldots+q+1)$. We have shown that these lines are precisely the ruling lines of $x$  $(t-1)$-reguli denoted $\R_1,\ldots,\R_x\subset\S$. By Lemma~\ref{222}, the reguli $\R_1,\ldots,\R_x$ correspond to a set of $x$ $\Fq$-lines of $\PG(1,q^t)$; we denote this set of $\Fq$-lines by $\F_i$ and note that each $\Fq$-line in $\F_i$ is contained in $\A$. That is, the replacement net
 $\A^i$ is associated with a family $\F_i$ of $x$ $\Fq$-lines which are contained in the Andr\'e set $\A$.

 As noted in Section~\ref{sec:back}, since $t$ is prime, for distinct $i,j$, every $(t-1)$-space in $\A^i$ meets every $(t-1)$-space in $\A^j$ in a unique point.
 In particular, if $\ell$ is a line contained in a $(t-1)$-space of $\A^i$, then $\ell$ is not contained in any $(t-1)$-space of $\A^j$, $j\neq i$. So two families $\F_i,\F_j$ with $i\neq j$ have no common $\Fq$-lines. Hence the families $\F_1,\ldots,\F_{t-1}$ contain a total of $(t-1)x$ $\Fq$-lines of $\A$.
As $t$ prime, by \cite[Prop 4.9]{csaj16} (as quoted above), the total number of $\Fq$-lines contained in $\A$ is $(t-1)x$, and so the families $\F_1,\ldots,\F_{t-1}$ partition the $\Fq$-lines of $\A$.
This proves part \ref{M110a}.

Let $P,Q\in\D$, so $[P],[Q]$ are elements of the $(t-1)$-spread $\S$. If $U$ is a point in $[P]$, then $U$ lies on a unique $(t-1)$-space $\Sigma_i$ in $\A^i$ for each $i\in\{1,\ldots,t-1\}$. Let $V_i=\Sigma_i\cap[Q]$, and $\ell_i=UV_i$. By part 1,  the set $\{P\in\li\,|\,[P]\cap\ell_i\neq\emptyset\}$  is an $\Fq$-line in family $\F_i$. That is, $P,Q$ lie on at least one $\Fq$-line in each family. As the total number of $\Fq$-lines in each family is $x$ which is equal to ${{q^{t-1}+\cdots+q+1}\choose2}/{{q+1}\choose2}$, it follows that $P,Q$ lie on a unique $\Fq$-line in each family. This proves part 2; the proof of part 3 is similar.
\end{proof}

These
 families of $\Fq$-lines are critical to our results on inherited arcs. In particular, the next result shows how the families $\F_i$ of $\Fq$-lines contained in $\D$ relate to lines of the Andr\'e planes $\P(\D^i)$.
Recall the definition of $\D^i$-blocks given in Construction~\ref{cons1}.

\begin{lemma}\label{odd001}
In $\PG(2,q^t)$,  $q\geq t$, $t\geq3$ prime, let $\D$ be an Andr\'e set of $\li$, and denote the families of $\Fq$-lines contained in $\D$ by $\F_1,\ldots,\F_{t-1}$.  Three non-collinear affine points $A,B,C$ lie in a common  $\D^i$-block for some $i\in\{1,\ldots,t-1\}$ if and only if the $\Fq$-line determined by the three points $AB\cap\li$, $AC\cap\li$, $BC\cap\li$  belongs to family $\F_i$.
\end{lemma}

\begin{proof} Let $A,B,C$ be three non-collinear affine points which lie in a common $\D^i$-block $\K$ for some $i\in\{1,\ldots,t-1\}$. Using the definition of $\D^i$-blocks given in Construction~\ref{cons1}, the set  $[\K]$ in $\PG(2t,q)$ is an affine $t$-space whose projective completion  meets $\si$ in a $(t-1)$-space $\Sigma$ which lies in the replacement net $\D^i$.
Let $\alpha$ be the plane determined by the points $[A],[B],[C]$, so $\alpha$ is contained in the projective completion of $[\K]$. Let $\alpha\cap\si=\ell$, then $\ell\subset\Sigma$.
The sets $\alpha,\ell$ correspond in $\PG(2,q^t)$ to sets denoted $\pi,b$ respectively.
By Result~\ref{222},  $\pi$ is an $\Fq$-plane that meets $\li$ in the $\Fq$-line $b$.
Moreover, $\pi$  is the unique $\Fq$-plane that is secant to $\li$ and contains $A,B,C$. So the three points $AB\cap\li$, $AC\cap\li$, $BC\cap\li$ lie in $b$.
As $\ell\subset\Sigma\in\D^i$, by Theorem~\ref{M110}, $b$ is an $\Fq$-line belonging to family $\F_i$.

Conversely, let $A,B,C$ be three non-collinear affine points and denote the $\Fq$-line  determined by the three points $AB\cap\li$, $AC\cap\li$, $BC\cap\li$ by $b$. Suppose $b$ belongs to family $\F_i$  for some $i\in\{1,\ldots,t-1\}$. Let $\pi$ be the unique $\Fq$-plane that is secant to $\li$ and contains $A,B,C$, so $\pi\cap\li=b$. By Result~\ref{222}, in $\PG(2t,q)$, $[b]$ is a $(t-1)$-regulus, and the plane $[\pi]$ meets $\si$ in a line $\ell$ which is a ruling line of the regulus $[b]$. By Theorem~\ref{M110}, the line $\ell$ lies in a unique $(t-1)$-space $\Sigma$ of the replacement net $\D^i$. The affine points of the $t$-space $\langle \pi,\Sigma\rangle$ correspond to   a $\D^i$-block of $\PG(2,q^t)$ that contains the three points $A,B,C$.
\end{proof}

\subsection{Preliminary counting}

It is straightforward to count the number of non-degenerate conics containing two distinct points, giving the following result.

\begin{result}\label{117} Let $P,Q$ be distinct points in $\PG(2,q)$.
There are exactly $q^2 (q- 1)$ non-degenerate conics containing $P$ and $Q$.
\end{result}


\begin{lemma}\label{104ta}
 In $\PG(2,q^t)$, $q$ odd, $t$ prime, let $\C$ be  a non-degenerate conic secant to $\li$, with $\C\cap\li=\{P,Q\}$.
 \begin{enumerate}
  \item If $t=2$, there are $q-1$ Baer sublines of $\li$ for which $\{P,Q\}$ are a conjugate pair; half contain only internal points of $\C$, half contain only external points of $\C$.
 \item If $t\geq 3$, there are $q-1$ Andr\'e sets of $\li$ with  transversal points $\{P,Q\}$; half contain only internal points of $\C$, half contain only external points of $\C$.
  \end{enumerate}
\end{lemma}

\begin{proof} Let $\li$ have equation $z=0$ and let $\C$ be a non-degenerate conic secant to $\li$. As $\PGammaL(3,q^t)$ is transitive on quadrangles, without loss of generality, assume $\C$   contains the points  $P=(1,0,0)$, $Q=(0,1,0)$ and   $(0,0,1)$. Hence $\C$ has equation $ f(x,y,z)=ayz+bxz+cxy=0$  with $a,b,c\in\Fqt$, $abc\neq0$.
If $t=2$, and $\D$ is a Baer subline of $\li$ for which the points $(1,0,0),(0,1,0)$ are a conjugate pair, then we can write
$\D=\{P_\theta=(1,\theta,0)\,|\,\theta\in\Fqt^*,\Nm(\theta)=\delta\}$ for some $ \delta\in\Fq^*$.
If $t\geq3$, and $\D$ is an Andr\'e set with transversal points $(1,0,0),(0,1,0)$, then we have
$\D=\{P_\theta=(1,\theta,0)\,|\,\theta\in\Fqt^*,\Nm(\theta)=\delta\}$ for some $ \delta\in\Fq^*$. That is, we can use this expression for $\D$ for all $t\geq2$.

Let $\tau$ be a multiplicative generator of $\Fqt^*$, then for $\delta\in\Fq^*$, we can write $\delta=\tau^{n(\frac{q^t-1}{q-1})}$ for some $n\in\{0,\ldots,q-2\}$. So if $\theta\in\Fqt^*$ with $\Nm(\theta)=\delta$, we can write $\theta=\tau^{n}\tau^{(q-1)i}$ for some $i\in\{0,\ldots,q^{t-1}+\cdots+q\}.$
So we can write the points of  $\D$ as $\D=\{Q_i=(1, \tau^{n}\tau^{(q-1)i},0)\,|\, i=0,\ldots,q^{t-1}+\cdots+q\}.$
To determine whether the point $Q_i\in\D$ is internal or external to $\C$, we look at $f(Q_i)$ where $f(x,y,z)=ayz+bxz+cxy=0$ is the equation  of $\C$: so we have  $f(Q_i)=c \tau^{n}\tau^{(q-1)i}$. Recall $c\neq0$ and
note that as $q$ is odd, $\tau^{(q-1)i}$ is a square in $\Fqt$ for all $i$.
So $f(Q_i)$ is a square in $\Fqt$  if and only if $c \tau^{n}$ is a square in $\Fqt$.
Hence either $f(Q_i)$ is a square  in $\Fqt$ for all points $Q_i\in\D$, or $f(Q_i)$ is a nonsquare  in $\Fqt$ for all points $Q_i\in\D$.
 Thus all the points of $\D$ are either all external to $\C$ or all internal to $\C$.  The result now follows as half the points on $\li\setminus\{P,Q\}$ are internal to $\C$, and half are external to $\C$.
 \end{proof}

\begin{lemma}\label{104tb}
 In $\PG(2,q^t)$, $q$ odd, $t$ prime.
 \begin{enumerate}
 \item\label{104tbpart1} If $t=2$, let $\D$ be a Baer subline of $\li$.
 \begin{enumerate}
 \item There are $\frac12  q^5 (q+1) ^2(q- 1)$ non-degenerate conics that meet $\li$ in two points contained in $\D$;  the non-conic  points of $\D$  are  all internal to half of these conics, and all external to the other half.
 \item There are $\frac12  q^5 (q+1) (q- 1)^2$ non-degenerate conics that   meet $\li$ in two points which are conjugate with respect to $\D$;  the points of $\D$ are  all internal to half of these conics, and all external to the other half.
 \end{enumerate}
  \item\label{104tbpart2}  If $t\geq 3$, let $\D$ be an Andr\'e set of $\li$ with transversal points $P,Q$. There are $q^{2t} (q^t- 1)$ non-degenerate conics containing $P$ and $Q$. The points of $\D$ are all internal to half of these conics, and all external to the other half.
  \end{enumerate}
\end{lemma}

\begin{proof}
In each case,  the number of non-degenerate conics follows from  Result~\ref{117}.
For the counts in part 1, let
 $\D$ be a Baer subline of $\li$ in $\PG(2,q^2)$. For part 1(a),  the number of choices for two distinct points $P,Q\in\D$ is ${q+1}\choose2$, so by Result~\ref{117}, the number of non-degenerate conics that contain two points of $\D$ is ${{q+1}\choose2}q^4(q^2-1)$. For part 1(b),
the number of choices for two points $P,Q$ a conjugate pair with respect to $\D$ is $\frac12(q^2-q)$, so by Result~\ref{117}, the number of non-degenerate conics that meet $\li$ in a pair of points that are conjugate with respect to  $\D$ is $\frac12(q^2-q)q^4(q^2-1)$. For the count in part 2, let  $\D$ be an Andr\'e set of $\li$ in $\PG(2,q^t)$, so $\D$ has  a unique pair of transversal points,  hence by Result~\ref{117}, the number of non-degenerate conics that contain the transversal points of  $\D$ is $q^{2t} (q^t- 1)$.

For the second statement in part 1(a), without loss of generality, assume $\D=\{P_k=(1,k,0)\,|\,k\in\Fq\cup\{\infty\}\}$, and let $\C$ be a non-degenerate conic containing the two points  $P=(1,0,0)$ and $Q=(0,1,0)$. So $\C$ has equation $f(x,y,z)=dz^2+ayz+bxz+cxy=0$ for some $a,b,c,d\in\Fqq$ with
 $c(ab - d c) \neq 0$. For $P_k\in\D\setminus\{P,Q\}$, that is, $k\in\Fq\setminus\{0\}$, we have $f(P_k)=ck$, which is a square in $\Fqq$ if and only if $c$ is a  square in $\Fqq$. Note that $c\neq0$, so exactly half the choices for $c$ are squares in $\Fqq$. Thus for half of the non-degenerate conics $\C$ containing $P,Q$, the numbers $f(P_k)$, $k\in\Fq\setminus\{0\}$ are all squares in $\Fqq$ (for the other half, they  are non-squares).
Thus for half of the non-degenerate conics $\C$ containing $P,Q$, the points of $\D\setminus\{P,Q\}$ are all internal to $\C$; and for the other half,  the points of $\D\setminus\{P,Q\}$ are all external to $\C$.

The proofs for the second statements in parts 1(b) and 2 can be combined.
Without loss of generality, assume  $P=(1,0,0)$, $Q=(0,1,0)$ and
$\D=\{P_\theta=(1,\theta,0)\,|\,\theta\in\Fqt^*,\Nm(\theta)=1\}$. Let $\tau$ be a multiplicative generator of $\Fqt^*$, then
 for $\theta\in \Fqt^*,\Nm(\theta)=1$, we have $\theta= \tau^{(q-1)i}$, $i=0,\ldots, q^{t-2}+\cdots+q$, so we can write $P_\theta=(1,\tau^{(q-1)i},0)$.
Let $\C$ be a non-degenerate conic that contains the points $P,Q$, so $\C$ has equation $f(x,y,z)=dz^2+ayz+bxz+cxy=0$ for some $a,b,c,d\in\Fqt$ with
 $c(ab - d c) \neq 0$.
Note that as $q$ is odd, $\tau^{(q-1)i}$ is a square in $\Fqt$ for all $i$.
So  for $P_\theta\in\D$, we have
$f(P_\theta)=f(1,\tau^{(q-1)i},0)=c \tau^{(q-1)i}$ is a square in $\Fqt$ if and only if $c$ is a square in $\Fqt$. Note that $c\neq0$, so exactly half the choices for $c$ are squares in $\Fqt$. Thus for half of the non-degenerate conics $\C$ containing $P,Q$, the points of $\D$ are all internal to $\C$; and for the other half,  the points of $\D$ are all external to $\C$.
 \end{proof}

\section{Inherited arcs in Andr\'e planes}

In this section we determine when a conic of $\PG(2,q^t)$, $t\geq3$ prime, inherits to an arc of an Andr\'e plane.
We need  the following result.

\begin{result}\label{resultodd} \cite[Thm 1]{Korch1}
In $\PG(2, q)$, $q$ odd, let $\C$ be a non-degenerate conic and $\ell$  a line. If $\ell$ is not a tangent of $\C$ and $\{P_1 , P_2 , P_3 \}$ is a set of points on $\ell$ containing either three or exactly one external points of $\C$, then there are exactly two triangles $\{A_1,A_2,A_3\}$ inscribed in $\C \setminus \ell$ such that $A_i A_j \cap \ell = P_k$, where $i, j, k$ is a permutation of $1, 2, 3$. In the other cases, for example, when the three points are internal points of $\C$, there is no $\{A_1 , A_2 , A_3 \}$ with this property.
\end{result}

In the next result, recall that $\D^i$-blocks are defined in Construction 2. 
 
\begin{lemma}\label{100t}
In $\PG(2,q^t)$, $q$ odd, $q\geq t$, $t\geq3$ prime, let $\C$ be a non-degenerate conic secant or exterior to $\li$ and $\D$ an Andr\'e set of $\li$.
 \begin{enumerate}
\item If $\D$ contains an external point of $\C$, then  for each $i\in\{1,\ldots,t-1\}$, there exists a $\D^i$-block that meets $\C$ is at least three points. 
\item If $\D$ does not contain an external point of $\C$, then every $\D^i$-block, $i\in\{1,\ldots,t-1\}$, meets $\C$ in at most two points.
\end{enumerate}
\end{lemma}

\begin{proof} 
First suppose $\D$ contains a point $P_1$ which is external to
$\C$. Let $i\in\{1,\ldots,t-1\}$. By Theorem~\ref{M110}(\ref{M110c}),
there are $q^{t-2}+\cdots+q+1$ $\Fq$-lines containing $P_1$ and
belonging to family $\F_i$. We show that there are at most two
$\Fq$-lines in $\F_i$ that contain $P_1$ and a point of $\C$.
If $|\C\cap\D|=0$, then every $\Fq$-line in $\F_i$  is disjoint from
$\C$. If $|\C\cap\D|=1$, then by Theorem~\ref{M110}(\ref{M110b}), there
is a unique $\Fq$-line in $\F_i$ containing $P_1$ and the point
$\C\cap\D$. If $|\C\cap\D|=2$, then let $\C\cap\D=\{P,Q\}$. Either the
$\Fq$-line containing the three points $P_1,P,Q$ lies in $\F_i$;
otherwise, by Theorem~\ref{M110}(\ref{M110b}), there is a unique
$\Fq$-line in $\F_i$ containing the points $P_1,P$; and a unique
$\Fq$-line in $\F_i$ containing the points $P_1,Q$.
In each case there are at most two $\Fq$ lines in $\F_i$ meeting $\C$. As  $q^{t-2}+\cdots+q+1\geq3$, there  exists an $\Fq$-line $b$ which belongs to family
$\F_i$, contains $P_1$ and is disjoint from $\C$.
 From the $q\geq3$ points in $b\setminus P_1$, we can choose  $P_2,P_3\in b$
  either both external or both internal to $\C$.
By Result \ref{resultodd}, there exist three points
$A_1,A_2,A_3\in\C\setminus\li$ such that $A_i A_j \cap \li = P_k$, where
$i, j, k$ is a permutation of $1, 2, 3$.
As $b$ is the unique $\Fq$-line determined by $P_1,P_2,P_3$, it follows from
Lemma~\ref{odd001} that
  there is a $\D^i$-block $\alpha$ containing  the affine points
$A_1,A_2,A_3$.

For part 2, suppose $\D$ contains no external points of $\C$; we give a proof  by contradiction.
Suppose that for some $i\in\{1,\ldots,t-1\}$, there is a $\D^i$-block $\alpha$ that meets $\C$ is three (affine) points 
$A_1,A_2,A_3$.
By Lemma~\ref{odd001}, the unique $\Fq$-line (denoted $b$) determined by the three points   $P_1=A_2A_3\cap\li$, $P_2=A_1A_3\cap\li$, $P_3=A_1A_2\cap\li$ belongs to family $\F_i$. In particular, $P_1,P_2,P_3$ lie in $\D$.
 As $\D$ contains no external points of $\C$, and $P_1,P_2,P_3$ are not in $\C$, they are internal points of $\C$.
This contradicts Result~\ref{resultodd}. Hence every $\D^i$-block, $i\in\{1,\ldots,t-1\}$, meets $\C$ in at most two points.
\end{proof}

%

We will use Lemma~\ref{100t} to prove our main results.  We first consider the case where exactly one Andr\'e net is replaced, that is, we consider Andr\'e planes of index 1.
We use the following notation. Recall that if $\D$ is an Andr\'e set of $\li$ of $\PG(2,q^t)$, then there are $t-1$ replacement Andr\'e nets denoted $\D^i$, $i\in\{1,\ldots,t-1\}$. For $i\in\{1,\ldots,t-1\}$, let $\P(\D^i)$ denote the Andr\'e plane obtained by performing an Andr\'e replacement  of $\PG(2,q^t)$ using the net $\D^i$. Suppose $\C$ is a non-degenerate conic of $\PG(2,q^t)$. Then the affine points of $\C$ correspond to a set of affine points of $\P(\D^i)$, we denote this set of affine points by $\C^i$. The next result answers the question of when $\C^i$ is an arc of the Andr\'e plane $\P(\D^i)$ for $i=1,\ldots,t-1$.

\begin{theorem}\label{200} In $\PG(2,q^t)$, $q$ odd, $q\geq t$, $t\geq3$ prime, let $\C$ be a non-degenerate conic and let $\D$ be an Andr\'e set of $\li$.  Perform an Andr\'e replacement using the Andr\'e replacement net $\D^i$; and let $\C^i$ denote the set of points in $\P(\D^i)$ corresponding to the affine points of $\C$. 
\begin{enumerate}
\item\label{200a}
 If $\C$ is tangent to $\li$, then $\C^i$ is not an arc of  the Andr\'e plane $\P(\D^i)$, $i=1,\dots,t-1$.
 \item\label{200c}
Suppose $\C$ is exterior to  $\li$.
\begin{enumerate} 
\item If $\D$ contains an external point of $\C$, then   $\C^i$ is not an arc of  the Andr\'e plane $\P(\D^i)$, $i=1,\dots,t-1$.
\item If $q$  satisfies (at least) one of the following: either \textup{(i)} $q\geq4t^2-8t+2$, or \textup{(ii)}
$q$ prime and $q>2t^2 -(4-2\sqrt{3})t+
(3-2\sqrt{3})$, or \textup{(iii)} $t=3$ and $q\geq 3$,
then $\C^i$ is not an arc of  the Andr\'e plane $\P(\D^i)$, $i=1,\dots,t-1$.
\end{enumerate}
   \item\label{200b}
 If $\C$ is secant to $\li$ with $\C\cap\li=\{P,Q\}$, then the following hold.
  \begin{enumerate}
 \item If $\D$ is one of the $\frac{q-1}2$ Andr\'e sets with transversal points $\{P,Q\}$ which contain only internal points of $\C$, then   $\C^i$ is  a  $(q^t-1)$-arc of  the Andr\'e plane $\P(\D^i)$, $i=1,\dots,t-1$.
Furthermore, the arc can be completed to an oval of $\P(\D^i)$  by adding $P$ and $Q$.
 \item Otherwise $\C^i$ is not an arc of  the Andr\'e plane $\P(\D^i)$, $i=1,\dots,t-1$.
 \end{enumerate}
   \end{enumerate}
\end{theorem}

\begin{proof}
Part \ref{200a} is proved by Korchm\'aros in \cite[Theorem 1]{Korch2}.

For part \ref{200c}, suppose $\C$ is exterior to  $\li$. If $\D$ contains an external point of $\C$, then by Lemma~\ref{100t}, there exists a $\D^i$-block   that meets $\C$ in at least three points. Hence there is a line of $\P(\D^i)$ that meets $\C^i$ in at least three points, so $\C^i$ 
 is not an arc of    $\P(\D^i)$  for any $i\in\{1,\ldots,t-1\}$. This proves part 2(a).  We show that if $q$ satisfies one of the conditions in part 2(b), then   $\D$ contains an external point of $\C$.
By Theorem~\ref{M110}(\ref{M110a}), the $\Fq$-lines contained in the Andr\'e set $\D$ can be partitioned into  $t-1$ non-empty families $\F_i$, $i=1,\ldots,t-1$. Let $\ell$ be an $\Fq$-line belonging to family $\F_i$ for some  $i\in\{1,\ldots,t-1\}$. If (i) $q\geq4t^2-8t+2$ or (ii) $q$ is prime and $q>2t^2 -(4-2\sqrt{3})t+
(3-2\sqrt{3})$,  we can use \cite[Thm 2.4]{Lav}, which shows that  $\ell$ contains at least one external point of $\C$. Hence $\D$ contains an external point of $\C$, so part 2(b) now follows from part 2(a).

For case (iii), suppose $t=3$, we show that $\D$ always contains an external point of $\C$. Let $\ell$ be an $\Fq$-line contained in $\D$ ($\ell$ exists by Theorem~\ref{M110}). By  \cite[Thm 2.4]{Lav}, if $q>14$, then $\ell$ contains at least one external point of $\C$. If $q=9,11,13$, then by the table in \cite[Section 3]{LavR}, $\ell$ contains at least one external point of $\C$.
 Hence if $q\geq9$, then $\D$ contains an external point of $\C$.
 If $q\in\{3,5,7\}$, then by  the table in \cite[Section 3]{LavR}, the number of $\Fq$-lines on $\li$ that contain entirely internal points of $\C$ is strictly less than $q^2+q+1$. By Theorem~\ref{M110}, the number of $\Fq$-lines contained in $\D$ is $2(q^2+q+1)$, so $\D$ contains an external point of $\C$. Hence by Lemma~\ref{100t}, there is a $\D^i$-block that meets $\C$ in at least three points, so
$\C^i$ is not an arc of  the Andr\'e plane $\P(\D^i)$  for any $i\in\{1,2\}$.


For part \ref{200b}, let $\C$ be a non-degenerate conic with $\C\cap\li=\{P,Q\}$. First suppose $\D$ is an Andr\'e set of $\li$ with transversal points $\{P,Q\}$; by Lemma~\ref{104ta} there are $q-1$ choices for $\D$, half contain only  internal points of $\C$  and half contain only external points of $\C$.
If $\D$ contains only external points of $\C$, then by Lemma~\ref{100t},  there is a $\D^i$-block that meets $\C$ in at least three points, so $\C^i$ is not an arc of  the Andr\'e plane $\P(\D^i)$  for any $i\in\{1,\ldots,t-1\}$.
If $\D$ contains only internal points of $\C$, then by Lemma~\ref{100t},  every $\D^i$-block, $i\in\{1,\ldots,t-1\}$,  meets $\C$ in at most two points, so every line of $\P(\D^i)$ meets $\C^i$ in at most two points; that is $\C^i$ is  an arc of  the Andr\'e plane $\P(\D^i)$ for all   $i\in\{1,\ldots,t-1\}$. We show that we can complete this arc to an oval. Let $i\in\{1,\ldots,t-1\}$ and let $\ell$ be a line of $\P(\D^i)$  through $P$ or $Q$ (distinct from $\li$), so $\ell$ is also a line of $\PG(2,q^t)$. So in $\PG(2,q^t)$ there is at most one point of $\C\setminus\{P,Q\}$ on $\ell$, hence in  $\P(\D^i)$  there is  at most one point of $\C^i$ on $\ell$. Hence, $\D_{i}(\C)\cup \{P,Q\}$ is an arc of size $(q^t+1)$, proving part \ref{200b}(a).

Next we consider the case where the two transversal points $E,F$ of $\D$ are not the two points $P,Q$, that is,  $|\{E,F\}\cap\{P,Q\}|\leq1$. We show that $\D$ contains an external point of $\C$ by computing a bound on the number of internal points of $\C$ on  $\D$.  Denote the $q-1$ Andr\'e sets with transversal points $P,Q$ by $\mathcal X_1,\ldots,\mathcal X_{q-1}$, so  we have $\D\neq \mathcal X_i$, $i=1,\ldots,q-1$.
By  Lemma~\ref{104ta} half of the $\mathcal X_i$ contain only  internal points of $\C$  and the other half contain only external points of $\C$. By Lemma~\ref{102t}, $|\D\cap\mathcal X_i|\leq2(q^{t-2}+\cdots +q+ 1)$.
 As $\{P,Q,\mathcal X_1,\ldots\mathcal X_{q-1}\}$ is a partition of the points of $\li$,  the number of internal points of  $\C$ on  $\D$ is at most
 $\frac{q-1}2 (2(q^{t-2}+\cdots +q+ 1))= q^{t-1}-1$.
 As $|\D\cap\{P,Q\}|\leq2$,  $\D$ contains at least $x=(q^{t-1}+q^{t-2}+\cdots+q+1)-(q^{t-1}-1)-2
 =q^{t-2}+q^{t-3}+\cdots+q $
external points of $\C$. As $t\geq 3$, we have $x>0$, that is, $\D$ contains an external point of $\C$. It follows from Lemma~\ref{100t} that  $\C^i$ is not an arc of  the Andr\'e plane $\P(\D^i)$ for any $i\in\{1,\ldots,t-1\}$, proving part \ref{200b}(b). (We note that if $t=2$, then $x=0$,  and it is indeed possible to have an inherited arc with $|\D\cap\C|=2$  in the $\PG(2,q^2)$ case (this is case 3b in Result~\ref{120}), although it is not possible in the $\PG(2,q^t)$, $t\geq3$ case.) This completes the proof of part \ref{200b}.
\end{proof}

We next consider the case where more than one Andr\'e net is replaced. We use the following notation.  
Let $\D_1,\ldots,\D_s$ be $s$ pairwise disjoint  Andr\'e sets of the line at infinity $\li$ of $\PG(2,q^t)$. In the Bruck-Bose setting, $\li$ corresponds to a regular $(t-1)$-spread $\S$ of $\PG(2t-1,q)$, and  each Andr\'e set $\D_j$ corresponds to an Andr\'e net denoted  $[\D_j]$. For $j\in\{1,\ldots,s\}$, denote the $t-1$    Andr\'e replacement nets for $[\D_j]$ by  $\D_j^{1},\ldots,\D_j^{t-1}$. Suppose we perform the following replacements:  for each $j\in\{1,\ldots,s\}$, we replace the net $[\D_j]$ with the net $\D_j^{i_j}$, for some $i_j\in\{1,\ldots,t-1\}$. That is, replace the regular spread $\S$ with the spread $\S'=\big(\S\setminus([\D_1]\cup\cdots\cup[\D_s])\big)\cup (\D_1^{i_1}\cup\cdots\cup\D_s^{i_s})$. Let $\P(\D_1^{i_1},\ldots,\D_s^{i_s})$ denote the  translation plane arising from the new spread $\S'$. As  noted in Section~\ref{sec2.3},  $\P(\D_1^{i_1},\ldots,\D_s^{i_s})$ is a subregular plane. Further $\P(\D_1^{i_1},\ldots,\D_s^{i_s})$ is an Andr\'e plane of index $s$ exactly when 
$\D_1,\ldots,\D_s$ share the same transversal points.
The next result looks at whether a conic of $\PG(2,q^t)$  inherits to an arc of the projective plane $\P(\D_1^{i_1},\ldots,\D_s^{i_s})$. We note that the following result shows that a necessary condition for a multiple Andr\'e replacement to yield an inherited arc is that $\D_1,\ldots,\D_s$ share the same transversal points, in which case the resulting plane is an Andr\'e plane of index s. 

\begin{theorem}\label{200G} In $\PG(2,q^t)$, $q$ odd, $q\geq t$, $t\geq3$ prime, let $\C$ be a non-degenerate conic and let $\D_1,\ldots,\D_s$ be $s$ pairwise disjoint  Andr\'e sets of $\li$.  Let $\C^{i_1,\ldots,i_s}$ denote the set of points in the plane $\P(\D_1^{i_1},\ldots,\D_s^{i_s})$ corresponding to the affine points of $\C$.
\begin{enumerate}
\item\label{200Ga}
 If $\C$ is tangent to $\li$, then $\C^{i_1,\ldots,i_s}$ is not an arc of  the  plane $\P(\D_1^{i_1},\ldots,\D_s^{i_s})$,  for all $i_j\in\{1,\dots,t-1\}$.
 \item\label{200Gc}
Suppose $\C$ is exterior to  $\li$.
\begin{enumerate} 
\item If $\D_1\cup\cdots\cup\D_s$ contains an external point of $\C$, then   $\C^{i_1,\ldots,i_s}$ is not an arc of   the plane $\P(\D_1^{i_1},\ldots,\D_s^{i_s})$,  for all $i_j\in\{1,\dots,t-1\}$.
\item If $q$  satisfies (at least) one of the following: either \textup{(i)} $q\geq4t^2-8t+2$, or \textup{(ii)}
$q$ prime and $q>2t^2 -(4-2\sqrt{3})t+
(3-2\sqrt{3})$, or \textup{(iii)} $t=3$ and $q\geq 3$,
then  $\C^{i_1,\ldots,i_s}$ is not an arc of   the plane $\P(\D_1^{i_1},\ldots,\D_s^{i_s})$,  for all $i_j\in\{1,\dots,t-1\}$.
\end{enumerate}
\item\label{200Gb}
 If $\C$ is secant to $\li$ with $\C\cap\li=\{P,Q\}$, then  the following hold. 
 \begin{enumerate}
 \item If   $\{\D_1,\ldots,\D_s\}$ is a subset of the set of $\frac{q-1}2$ Andr\'e sets with transversal points $\{P,Q\}$ which contain only internal points of $\C$, then    $\C^{i_1,\ldots,i_s}$ is  a  $(q^t-1)$-arc of  the Andr\'e plane $\P(\D_1^{i_1},\ldots,\D_s^{i_s})$, for all $i_j\in\{1,\dots,t-1\}$.
Furthermore, the arc can be completed to an oval   by adding $P$ and $Q$.
 \item Otherwise $\C^{i_1,\ldots,i_s}$ is not an arc of  the plane $\P(\D_1^{i_1},\ldots,\D_s^{i_s})$,  for all $i_j\in\{1,\dots,t-1\}$.
 \end{enumerate}
   \end{enumerate}
\end{theorem}

\begin{proof}
Part \ref{200Ga} is proved by Korchm\'aros in \cite[Theorem 1]{Korch2}.

For part \ref{200Gc}, suppose $\C$ is exterior to  $\li$.
Suppose $\D_1\cup\cdots\cup\D_s$ contains an exterior point of $\C$. There there is some $j\in\{1,\ldots,s\}$ for which $\D_j$ contains an exterior point of $\C$. So by Lemma~\ref{100t}, for each $i_j\in\{1,\ldots,t-1\}$, there is a $\D_j^{i_j}$-block that meets $\C$ in at least three points. This  $\D_j^{i_j}$-block is a  line of $\P(\D_1^{i_1},\ldots,\D_s^{i_s})$ that meets $\C^{i_1,\ldots,i_s}$ in at least three points. Hence $\C^{i_1,\ldots,i_s}$ 
 is not an arc of  $\P(\D_1^{i_1},\ldots,\D_s^{i_s})$  for any $i_j\in\{1,\ldots,t-1\}$.
This proves part 2(a). If  $q$ satisfies the  bound given in part 2(b), then the proof of Theorem~\ref{200} shows that  each $\D_j$ contains an external point of $\C$, and so part 2(b) follows from part 2(a).

For part \ref{200Gb}, suppose $\C$ is secant to $\li$. Suppose each of the Andr\'e nets $\D_j$, $j\in\{1,\ldots,s\}$ has the same transversal points  $\{P,Q\}=\C\cap\li$. If some $\D_j$ contains an external point of $\C$, then by Lemma~\ref{100t},  for each $i_j\in\{1,\ldots,t-1\}$, there is a $\D_j^{i_j}$-block that meets $\C$ in at least three points.  Hence $\C^{i_1,\ldots,i_s}$ 
 is not an arc of  $\P(\D_1^{i_1},\ldots,\D_s^{i_s})$  for any $i_j\in\{1,\ldots,t-1\}$.
If however, each of $\D_1,\ldots,\D_s$ contains only internal points of $\C$,  then by Lemma~\ref{100t},  for each $i_j\in\{1,\ldots,t-1\}$,  each $\D_j^{i_j}$-block  meets $\C$ in at most two points.  Hence $\C^{i_1,\ldots,i_s}$ 
 is  an arc of  $\P(\D_1^{i_1},\ldots,\D_s^{i_s})$  for any $i_j\in\{1,\ldots,t-1\}$. Similar to the proof of Theorem~\ref{200}, this arc can be completed to an oval.  The remaining case to consider is when there exists   $j\in\{1,\ldots,s\}$ for which $\D_j$ has transversal points not equal to $P,Q$. In this case the proof of 
Theorem~\ref{200} shows that $\D_j$ contains an external point of $\C$. It follows from Lemma~\ref{100t} that $\C^{i_1,\ldots,i_s}$ 
 is not an arc of  $\P(\D_1^{i_1},\ldots,\D_s^{i_s})$  for any $i_j\in\{1,\ldots,t-1\}$.
\end{proof}

\section{Counting inherited arcs}

A natural question to ask is how many conics of $\PG(2,q^t)$ inherit to arcs. This has not been computed for the Hall plane (the case when $t=2$) so we first look at
 how many conics from $\PG(2,q^2)$ inherit to arcs of the Hall plane.

\begin{theorem}\label{119}
In $\PG(2,q^2)$, $q$ odd, $q>5$.
\begin{enumerate}
\item Let $\D$ be a Baer subline of $\li$ and derive $\PG(2,q^2)$ using $\D$ to get the Hall plane ${\mathcal H}(\D)$. There are exactly $\frac12   q^6 (q ^2- 1) $ conics of $\PG(2,q^2)$ that inherit to an arc of ${\mathcal H}(\D)$; $\frac14 q^5 (q + 1) (q - 1)^2$ of these can be completed to ovals of the Hall plane.
\item Let $\C$ be a non-degenerate conic secant to $\li$. There are exactly $q$ derivation sets $\D$ of $\li$ such that deriving $\PG(2,q^2)$ using $\D$ leads to an inherited arc of the associated Hall plane ${\mathcal H}(\D)$.
\end{enumerate}
\end{theorem}

\begin{proof}
To prove part 1, let $\D$ be a Baer subline of $\li$ and derive $\PG(2,q^2)$ using $\D$ to get the Hall plane ${\mathcal H}(\D)$. If $\C$ is a non-degenerate conic, then
Result~\ref{120} lists the cases where $\D(\C)$ is an arc of ${\mathcal H}(\D)$.
So we need to count the conics $\C$ secant to $\li$ with $\C\cap\li=\{P,Q\}$ where either (a)
$P,Q\notin\D$, $P,Q$ are conjugate with respect to $\D$, and all points of $\D$ are internal with respect to $\C$; or (b) $P,Q\in\D$ and all points of $\D\setminus\{P,Q\}$ are internal with respect to $\C$.
 The number of conics of type (a) is  $\frac14 q^5 (q + 1) (q - 1)^2$ by Lemma~\ref{104tb}(1b). Further, the corresponding inherited arcs  can be completed to  ovals of the Hall plane by Result~\ref{120}.
 The number of conics of type (b) is  $\frac14 q^5 (q + 1)^2 (q - 1)$ by Lemma~\ref{104tb}(1a). Further, the corresponding inherited arcs  are  complete $(q^2-1)$-arcs of the Hall plane by Result~\ref{120}.
Summing these two numbers gives the result of part 1.

To prove part 2, let $\C$ be a conic secant to $\li$ with $\C\cap\li=\{P,Q\}$.  By Result~\ref{120}, we need to count Baer sublines $\D$ which either (c) $P,Q\notin\D$, $P,Q$ are conjugate with respect to $\D$ and all points of $\D$ are internal  with respect to $\C$; or (d) $P,Q\in\D$ and all points of $\D\setminus\{P,Q\}$ are internal with respect to $\C$.

The number of Baer sublines of type (c) is $\frac{q-1}2$
by Lemma~\ref{104ta}.
To count the number $y$ of Baer sublines of type (d)
 we  count in two ways the number $x$ of pairs $(b,\O)$ where $b$ is a Baer subline of $\li$ and $\O$ is a non-degenerate conic with $|b\cap\O|=2$ and all points of $b\setminus\{\O\}$ internal with respect to $\O$.   By Result \ref{117}, the number of non-degenerate conics secant to $\li$ is
${q^2+1 \choose 2} q^4 (q^2- 1)
= \frac12 (q^2+1 ) q^6 (q^2- 1).$
 So we have $x=
  \frac12 (q^2+1 ) q^6 (q^2- 1)    \cdot  y$.
  The number of Baer sublines on $\li$ is $q(q^2+1)$.
The number of choices for $\O$ follows from  Lemma~\ref{104tb}(\ref{104tbpart1}a), so we have $x=(q^2+1)q  \cdot \frac14  q^5 (q+1)^2  (q - 1)$.
Equating the two expressions for $x$ and solving gives $y=\frac12   (q+1)$.  Summing the number of Baer sublines of types (c) and (d) gives part 2.
\end{proof}

The next result counts inherited arcs
in Andr\'e planes of index 1.  

\begin{theorem}\label{118t} In $\PG(2,q^t)$, $q$ odd, $q\geq t$, $t\geq3$ prime, suppose $q$  satisfies (at least) one of the following: either $q\geq4t^2-8t+2$; or
$q$ prime and $q>2t^2 -(4-2\sqrt{3})t+
(3-2\sqrt{3})$; or $t=3$ and $q\geq 3$.
 Let $\D$ be an Andr\'e set of $\li$ and perform an Andr\'e replacement using the Andr\'e replacement net $\D^i$ to get the Andr\'e plane $\P(\D^i)$ of index 1.
 For $i\in\{1,\ldots,t-1\}$, there are exactly $\frac12 q^{ 2 t} (q^t- 1) $ conics of $\PG(2,q^t)$ which inherit to arcs of the Andr\'e plane $\P(\D^i)$.
\end{theorem}


\begin{proof} Denote the transversal points of $\D$  by  $E$, $F$. By Theorem~\ref{200}, the only conics which inherit to arcs of $\P(\D^i)$ are the non-degenerate conics that contain $E,F$ and for which $\D$ consists of internal points of $\C$.
By Lemma \ref{104tb}(2),
the number of non-degenerate conics that contain $E,F$ is $q^{2 t} (q^t- 1)$ and
for half of these conics, $\D$ consists of external points only; and for half of these conics, $\D$ consists of internal  points only. Hence  there are exactly $\frac12 q^{2 t} (q^t- 1) $ conics of $\PG(2,q^t)$ which inherit to arcs of $\P(\D^i)$.
\end{proof}

Note that this generalises to count $\Fq$-conics inherited to arcs of Andr\'e planes of index $s$ by using Theorem~\ref{200G} instead of Theorem~\ref{200}.

\bigskip

{\bfseries Acknowledgments} We would like to thank the anonymous referee for thoughtful comments which helped improve this article.

{\bfseries AI Declaration}
No generative AI was used in any part of this project.

%
%
%

\begin{thebibliography}{99}

\bibitem{andr54} J.\ Andr\'e. \"Uber nicht-Desarguessche Ebenen mit
 transitiver Translationgruppe. {\em Math.\ Z.},  60 (1954)
 156--186.
%



\bibitem{BHJ-even} S.G. Barwick, A.M.W. Hui and W.-A. Jackson.  Inherited conics in  Andr\'e planes of even order. 
Preprint. http://arxiv.org/abs/2606.14345.


%
\bibitem{BJ2012}
S.G. Barwick and W.-A. Jackson.
Sublines and subplanes of PG$(2,q^3)$ in the Bruck-Bose representation in PG$(6,q)$.
{\em Finite Fields Appl.},
18 (2012)
93--107.


\bibitem{barmar}
S.G. Barwick and D.J. Marshall.
Conics and multiple derivation.
{\em Discrete Math.},
312 (2012)
1623--1632.
%
\bibitem{BKNS}
A. Blokhuis, I. Kov\'{a}cs, G.P. Nagy and T. Sz\H{o}nyi.
Inherited conics in Hall planes.
\emph{Discrete Math.},
342 (2019)
1098--1107.
%

\bibitem{bose47} R.C. Bose. Mathematical theory of the symmetrical factorial design. \emph{Sankhy\={a}}, 8 (1947) 107--166.

\bibitem{magma}
W. Bosma, J. Cannon, and C. Playoust.
The Magma algebra system. I. The user language,
{\em J. Symbolic Comput.},
    24 (1997)
    235--265.
%
\bibitem{bruc69} R.H.\ Bruck.  Construction problems of finite
projective planes.  {\em Conference on Combinatorial Mathematics
and its Applications},  University of North Carolina Press, (1969)
426--514.
%
\bibitem{bruck70}
R.H. Bruck. Some relatively unknown ruled surfaces in projective space.
Arch. Inst. Grand-Ducal Luxembourg Sect.
\emph{Sci. Nat. Phys. Math. N.S.}
(1970) 361--376.
%
\bibitem{bruc73a}
R.H. Bruck. Circle geometry in higher dimensions.
{\em A Survey of Combinatorial Theory},
eds. J.N. Amsterdam Srivastava et al,
(1973) 69--77.
%
\bibitem{bruc73b}
 R.H. Bruck.
 Circle geometry in higher dimensions. II.
 {\em Geom. Dedicata},
 2 (1973) 133--188.
%
%
\bibitem{bruc64}
R.H. Bruck and R.C. Bose.
The construction of  translation planes from projective spaces.
{\em J.\ Algebra},
 1 (1964)
 85--102.
%
%
%
\bibitem{casa22}
V. Casarino, G. Longobardi and C. Zanella.
Scattered linear sets in a finite projective line and translation planes.
\emph{Linear Algebra Appl.},
650 (2022)  286--298.


\bibitem{Cher}
W. Cherowitzo.
The classification of inherited hyperconics in Hall planes of even order.
\emph{European J. Combin.},
31 (2010)
81--86.

\bibitem{csaj}
B. Csajb\'ok, G. Marino and F. Zullo.
New maximum scattered linear sets of the projective line.
\emph{Finite Fields App.},
 54 (2018) 133--150.
%
\bibitem{csaj16}
B. Csajb\'ok and C. Zanella.
On scattered linear sets of pseudoregulus type in $\PG(1,q^t)$.
\emph{Finite Fields App.},
    41 (2016) 34--54.
%

%

%
%
\bibitem{dona14}
G. Donati and N. Durante.
Scattered linear sets generated by collineations between pencils of lines,
{\em J.  Algebraic. Comb.},
40 (2014)
1121--1134.

\bibitem{dover} J.M. Dover. Computational enumeration of Andr\'e planes. \emph{J. Comb. Math. Comb. Comput.,} 120 (2024) 417--434. 
%
%
%
%
%
\bibitem{GlynnStein}
D.G. Glynn and G.F. Steinke.
On conics that are ovals in a Hall plane.
\emph{Europ. J. Combin.},
14 (1993), 521--528.
\bibitem{grim}
G.G. Grimaldi, S. Gupta, G. Longobardi, R. Trombetti.
A geometric characterization of known maximum scattered linear sets of $\PG(1, q^n)$.
http://arxiv.org/abs/2405.01374v1


\bibitem{Hirsch1}
J.W.P. Hirschfeld.
\emph{Projective Geometries over Finite Fields}, second edition,
Oxford University Press, Oxford (1998).




 \bibitem{JJ2008}
 V. Jha and N.L. Johnson.
 A new class of translation planes constructed by hyper-regulus replacement.
 \emph{J. Geom.},
 90 (2008)
 83--99.
%
%
\bibitem{johnson2003}
N.L. Johnson.
Hyper-reguli and non-Andr\'e quasi subgeometry partitions of projective spaces.
\emph{J. Geom.},
 78 (2003)
 59--82.
%
%
\bibitem{john07}
N.L. Johnson, V. Jha and M. Biliotti.
{\em Handbook of Finite
  Translation Planes.} Chapman and Hall/CRC 2007.
%
%
%
%
\bibitem{Korch1}
G. Korchm\'{a}ros.
Ovali nei piani di hall di ordine dispari.
\emph{Atti Accad. Naz. Lincei, Rend.},
56 (1974)
315--317.
%
%
%
\bibitem{Korch2}
G. Korchm\'aros.
Inherited arcs in finite affine planes.
\emph{J. Combin. Theory Ser. A},
42 (1986)
140--143.
%
%
%
%
\bibitem{Lav}
M. Lavrauw.
Sublines of prime order contained in the set of internal points of a conic.
\emph{Des. Codes Cryptogr.},
38 (2006)
113--123.

%
%
\bibitem{LavR}
M. Lavrauw and M. Rodgers.
Classification of 8-dim rank 2 commutative semifields.
\emph{Adv. Geom.},
19 (2019)
57--64.
\bibitem{lavr15}
M. Lavrauw, J, Sheekey and C. Zanella.
On embeddings of minimum distance of $\PG(n,q)\times\PG(n,q)$.
\emph{Des. Codes Cryptogr.},
74 (2015)
427--440.
\bibitem{lavr10}
M. Lavrauw and G. Van de Voorde.
On linear sets on a projective line.
{\em Des. Codes Cryptogr.},
56 (2010)
89--104.


%
\bibitem{lavr14a}
M. Lavrauw and G. Van de Voorde.
Field Reduction and linear sets in finite geometry.
\emph{Topics in Finite Fields,} AMS Contemporary Math., ed
by G. Kyureghyan, G. L. Mullen, and A. Pott (2015).

%
%

%
%
%
%


%
%
%
%
\bibitem{LMPT14}
G. Lunardon, G. Marino, O. Polverino and R. Trombetti.
Maximum scattered linear sets of pseudoregulus type
and the Segre variety ${\cal S}_{n,n}$.
{\em  J. Algebraic. Comb.},
39 (2014)
807--831.

%
%

\bibitem{luna01}
G. Lunardon and O. Polverino.
Blocking sets and derivable partial spreads.
\emph{J. Algebr. Comb.},
14 (2001)
49--56.
%
%
%

%
\bibitem{OKeefePasPen}
C.M. O'Keefe, A.A. Pascasio, and T. Pentilla.
Hyperovals in Hall planes. \emph{Eur. J. Combin.},
13 (1992)
195--199.
%
\bibitem{OKeefePas}
  C.M. O'Keefe and A.A. Pascasio.
Images of conics under derivation.
  \emph{Discr. Math.},
  151 (1996)
   189--199.

%
\bibitem{ostrom70}
T.G. Ostrom.
Finite translation planes.
Lecture Notes in Mathematics. 158,
Springer-Verlag, New York, 1970.

%
\bibitem{ostrom}
T.G. Ostrom.
Hyper-reguli.
\emph{J. Geom.},
48 (1993)
157--166.
%
%
\bibitem{segre}
B. Segre.
Teoria di Galois, fibrazioni proiettive e geometrie non desarguesiane.
\emph{Ann. Mat. Pura Appl.},
 64 (1964)
1--64.
%
%
\bibitem{sheek2018}
J. Sheekey.
A new family of linear maximum rank distance codes.
\emph{Adv. Math. Commun.},
10 (2016)
475--488.
%

%
%
\bibitem{Szon}
T. Sz\H{o}nyi.
Arcs and $k$-sets with large nucleus set in Hall planes.
\emph{J. Geom.},
34 (1989)
187--194.




\end{thebibliography}
\end{document}